\documentclass[11pt]{article}
\usepackage{amsmath, amsthm, latexsym, amssymb,url}
\usepackage[usenames,dvipsnames]{color}
\usepackage{amsfonts}
\usepackage{xypic}
\usepackage{epsfig}
\usepackage{titlesec}
\usepackage{amsmath,color,graphicx,amscd,amsfonts,amsthm,amssymb}
\usepackage{yhmath}
\usepackage{mathdots}

\titleformat{\section}{\normalfont\Large\bfseries}{\S\thesection}{1em}{}
\input xy
\xyoption{all}

\newcommand{\shrinkmargins}[1]{
  \addtolength{\textheight}{#1\topmargin}
  \addtolength{\textheight}{#1\topmargin}
  \addtolength{\textwidth}{#1\oddsidemargin}
  \addtolength{\textwidth}{#1\evensidemargin}
  \addtolength{\topmargin}{-#1\topmargin}
  \addtolength{\oddsidemargin}{-#1\oddsidemargin}
 \addtolength{\evensidemargin}{-#1\evensidemargin}
  }
\shrinkmargins{.7}
\theoremstyle{plain}
\newtheorem{theorem}{Theorem}[section]

\newtheorem{lemma}[theorem]{Lemma}
\newtheorem{proposition}[theorem]{Proposition}
\newtheorem{question}[theorem]{Question}

\newtheorem*{teo}{Theorem}

\theoremstyle{remark}

\theoremstyle{definition}

\def \C { \mathbb{C}}

\def \ker { \text{Ker}}
\def \R { \mathbb{R}}

\begin{document}

\title{On a characterization of path connected topological fields}
\date{}
\author{Xavier Caicedo \hspace{1.5in} Guillermo Mantilla-Soler}
\maketitle

\begin{abstract}
The aim of this paper is to give a characterization of path connected
topological fields, inspired by the classical Gelfand correspondence between
a compact Hausdorff topological space $X$ and the space of maximal ideals of
the ring of real valued continuous functions $C(X,\R)$. More explicitly, our
motivation is the following question: What is the essential property of the
topological field $F=\R$ that makes such a correspondence valid for all
compact Hausdorff spaces? It turns out that such a perfect correspondence
exists if and only if $F$ is a path connected topological field.
\end{abstract}

\section{Introduction}

Let us recall briefly what happens with the real field. Let $X\neq \emptyset 
$ be a compact Hausdorff topological space, and let $C(X,\R)$ be the set of
continuous functions from $X$ to $\R$. Since evaluating at a point $x_{0}\in
X$ is a surjective ring homomorphism $\mathrm{ev}_{x_{0}}:C(X,\R)\rightarrow %
\R;\ f\mapsto f(x_{0})$, the kernel of such morphism, $I_{\R
}(x_{0}):=\ker (\mathrm{ev}_{x_{0}})$, is a maximal ideal of $C(X,\R)$. By
endowing the set of maximal ideals $\mathrm{Max}(C(X,\R))$ with the Zariski
topology, Gelfand's theorem tells us that this correspondence is actually a
homeomorphism, explicitly:

\begin{teo}
\label{Gelfand} Let $X$ be a non-empty compact Hausdorff
topological space. Then the map 
\begin{equation*}
\begin{array}{cccccc}
I_{\R}: & X & \rightarrow & \mathrm{Max}(C(X,\R)) &  &  \\ 
& x_{0} & \mapsto & I_{\R}(x_{0}) &  & 
\end{array}%
\end{equation*}
is a homeomorphism.
\end{teo}

In other words the algebraic structure of the ring $C(X,\R)$ completely
characterizes the space $X$. What is so special about $\R$ that makes such a
powerful characterization possible? Could we ensure the same kind of result
if we replace $\R$ by another topological field $F$?\newline

Let $F$ be a topological field. Let $X\neq \emptyset $ be a compact
Hausdorff topological space, and let $C(X,F)$ be the set of continuous
functions from $X$ to $F$. As in the real case given $x_{0}\in X$ if $%
I_{F}(x_{0})$ is the kernel of the evaluation morphism at the point $x_{0}$
then $I_{F}(x_{0})\in \mathrm{Max}(C(X,F))$. The question we address is the
following:

\noindent What conditions on a topological field $F$ are sufficient and
necessary such that for every compact Hausdorff topological space $X$ the
function 
\begin{equation*}
\begin{array}{cccccc}
I_{F}: & X & \rightarrow & \mathrm{Max}(C(X,F)) &  &  \\ 
& x_{0} & \mapsto & I_{F}(x_{0}) &  & 
\end{array}%
\end{equation*}
is a homeomorphism? The answer is given by the main result of the paper:

\begin{teo}[cf. Theorem \protect\ref{TheRealMain}]
Let $F$ be topological field. Then the Gelfand map $I_{F}$ is a
homeomorphism for every compact Hausdorff topological space $X$ if and only
if $F$ is path connected.
\end{teo}

The continuity of the Gelfand map holds more generally for any
topological field. Injectivity and bicontinuity will be shown to be
equivalent to path connectedness. And, remarkably, surjectivity will be seen
to hold for any topological field.

\ \ \ 

Our proof of the later fact relays on the famous classification of
non-discrete locally compact fields (\cite{Jacobson:1936}, \cite{WZ:1958},
or\ \cite{Bourbaki:1964}), and the observation (Proposition \ref{NoAlgClosed}%
) that for a non algebraically closed field $F$ there is a polynomial
function $\psi :F^{2}\rightarrow F$ such that 
\begin{equation*}
\psi ^{-1}(\{0_{F}\})=(0_{F},0_{F}).
\end{equation*}%
This polynomial must have the form $\psi (x,y)=x\phi _{1}(x,y)+y\phi
_{2}(x,y)$, and it generalizes the function $\psi (x,y)=x^{2}+y^{2}$ for $%
\mathbb{R}$. In the algebraically closed case such polynomial can not exist,
but surjectivity of $I_{\mathbb{C}}$ may be shown utilizing the function $%
\psi (x,y)=x\overline{x}+y\overline{y}$ which has the displayed property. By
analyzing the cases we have in hand, we suspect that there is for any path
connected topological field $F$ a similar function $\psi (x,y)=x\phi
_{1}(x,y)+y\phi _{2}(x,y)$ with $\phi _{1},$ $\phi _{2}:F^{2}\rightarrow F$
continuous, but have not been able to prove this.

\ \ \ 

The subject of rings of continuos functions with values in a field is
classic and has been studied by several authors, see for instance \cite%
{Bachman:1975}, \cite{GelfKolm:1939}, \cite{Niels:1987}. A survey article on
the subject containing a great deal of results and many references is \cite%
{Vechtomov:1996}. Although our characterization of path connected fields
given by Theorem \ref{TheRealMain} is quite natural, it seems to have been
missed in previous characterizations, see \cite[Theorem 8]{Niels:1987}.

\subsection{Path connected fields}

In this paper a \textit{topological field} is a field $F$ which is also a
topological space in which the operations are continuous and the points are
closed. Since the topology must be regular by general properties of
topological groups, $F$ must be Hausdorff. We denote by $0_{F}$ and $1_{F}$
the zero the and identity of $F$ respectively. Recall that a topological
space $X$ is \textit{path connected} if for every two points $x,y\in X$
there is a continuous path $\gamma :[0,1]\rightarrow X$ joining $x$ and $y,$
and it is \textit{arcwise connected} if $\gamma $ may be always chosen as a
homeomorphism between $[0,1]$ and its image. The following observation shows
that the various forms of path connectivity are equivalent in fields.

\begin{lemma}
\label{conexoCharCero} Let $F$ be a topological field. The following are
equivalent:

\begin{itemize}
\item[(i)] There is a continuous path $\gamma :[0,1]\rightarrow F$ joining $%
0_{F}$ and $1_{F}$.

\item[(ii)] $\ F$ is contractible.

\item[(iii)] $F$ is path connected.

\item[(iv)] $F$ is arcwise connected.
\end{itemize}
\end{lemma}

\begin{proof}Path and arcwise connectendness are known to be equivalent
in Hausdorff spaces (\cite{Willard:1970}, Corollary 31.6). It is enough then
to show the implications (i) $\Longrightarrow $ (ii) and (iii) $
\Longrightarrow $ (i). Given a path $\gamma $ joining $0_{F}$ and $1_{F}$,
the function $H:[0,1]\times F\rightarrow F;$ $(t,\lambda )\mapsto (1-\gamma
(t))\lambda $ gives the result. For the later case, given a path $\gamma
_{1} $ joining two different points $a,b\in F$ the function $\displaystyle
\gamma (t)=\frac{\gamma _{1}(t)-a}{b-a}$ gives a path between $0_{F}$ and $
1_{F}$. \end{proof}

\section{A first characterization}

We prove in this section a first approximation to our main result (Theorem %
\ref{main} below) by following essentially the lines of the proof that
the map $I_{\R}$ is injective and bicontinuous for any compact Hausdorff $%
X $.

Continuity of $I_{\R}$ follows from definition of the Zariski topology, and
from the fact that $0$ is closed in $\R$. This holds for every topological
field $F$ and every space $X$ (not necessarily compact or Hausdorff).

\begin{lemma}
\label{continua} For any space $X$ and any topological field the Gelfand map 
$I_{F}:X\rightarrow \mathrm{Max}(C(X,F))$ is continuous.
\end{lemma}

\begin{proof}By definition of the Zariski topology a closed subset of $
\mathrm{Max}(C(X,F))$ is of the form 
\begin{equation*}
C_{S}:=\{M\in \mathrm{Max}(C(X,F))\mid S\subseteq M\}
\end{equation*}
for some $S\subseteq C(X,F).$ Continuity of $I_{F}$ follows since $\{0_{F}\}$ is
closed and 
\begin{equation*}
I_{F}^{-1}(C_{S})=\{x:S\subseteq I_{F}(x)\}=\bigcap_{f\in S}f^{-1}(0_{F}).
\end{equation*} \end{proof}

Injectivity of $I_{\R}$ follows from the fact that $\R$ is path connected
and compact Hausdorff spaces are completely regular. These ideas can be
generalized as follows:

\begin{lemma}
\label{UnoaUno}Let $F$ be a topological field. Then $F$ is path connected if
and only if for any compact Hausdorff space $X$ the Gelfand map $%
I_{F}:X\rightarrow \mathrm{Max}(C(X,F))$ is injective.
\end{lemma}

\begin{proof} Assume $I_{F}$ is injective for the space $X=[0,1]$ then $
I_{F}(0)\neq I_{F}(1)$, pick $f\in I_{F}(0)\setminus I_{F}(1)$ then $
f(0)=0_{F}$ and $f(1)\neq 0_{F}.$ Hence, $\gamma =$ $f(1)^{-1}f$ defines a
path connecting $0_{F}$ and $1_{F}$. Reciprocally, assume there is a continuous path $\gamma $ as described.
Let $x\neq y\in X$. Since $X$ is completely regular it follows that there is
a continuous function $f:X\rightarrow \lbrack 0,1]$ such that $f(x)=0$ and $
f(y)=1$. Notice that $\gamma \circ f(x)=0_{F}$ and $\gamma \circ f(y)=1_{F}$
. In particular, $\gamma \circ f\in I_{F}(x)\setminus I_{F}(y)$ hence $I_{F}$
is injective. \end{proof}

Path connectedness of $F$ implies also that the image of the embedding is
Hausdorff for any completely regular space. Recall that an open basis for $%
\mathrm{Max}(C(X,F))$ is given by the sets $D(f)=\{M:f\notin M\}.$

\begin{lemma}
\label{1-2-1} Let $F$ be a path connected topological field. Then the image
of $I_{F}$ is Hausdorff for every compact Hausdorff space $X$.
\end{lemma}

\begin{proof} To separate $I_{F}(x_{0})\neq I_{F}(x_{1})$ in $\mathrm{Max}
(C(X,F))$ it is enough to show that there are $f,g\in C(X,F)$ such that $
f(x_{0})g(x_{1})\neq 0_{F}$ (thus, $I_{F}(x_{0})\in D(f),$ $I_{F}(x_{1})\in
D(g))$ and $fg\equiv 0_{F}$ (thus $D(f)\cap D(g)=\emptyset )$ Since $
x_{0}\neq x_{1}$ and $X$ is Hausdorff there are disjoint open sets $U_{i},$ $
i=0,1$, such that $x_{i}\in U_{i}$. Letting $C_{i}:=X\setminus U_{i}$ we see
that the $C_{i}$ are closed subsets of $X$. Therefore, by complete
regularity of $X$ there are $\widetilde{f},\widetilde{g}\in C(X,[0,1])$ such
that 
\begin{equation*}
\widetilde{g}(x_{1})=1=\widetilde{f}(x_{0})\ \mathrm{and}\ \widetilde{g}
(C_{0})=0=\widetilde{f}(C_{1}).
\end{equation*}%
Fix a continuous path $\gamma :[0,1]\rightarrow F$ joining $0_{F}$ and $
1_{F} $, and let $f:=\gamma \circ \widetilde{f}$ and $g:=\gamma \circ 
\widetilde{g} $. By the above $f(x_{0})=1_{F}$ and $g(x_{1})=1_{F}$. Let $
x\in X$, then $x\in C_{0}$ or $x\in C_{1}$, which by construction implies
that $f(x)g(x)=0_{F}.$ \end{proof}

By standard topology a continuous injection from a compact space in a
Hausdorff space is bicontinuous. Hence, combining lemmas 2.1, 2,1 and 2.3,
we obtain a first characterization.

\begin{theorem}
\label{main}Let $F$ be a topological field. Then $F$ is path connected if
and only if for every compact Hausdorff space $X$ the Gelfand map $I_{F}$
induces a homeomorphism between $X$ and its image.
\end{theorem}

\section{Surjectivity}

We will show in this section that the Gelfand map is surjective for any
compact Hausdorff space $X$ and any topological field $F.$ In this case the
techniques for $I_{\R}$ do not generalize fully. Surjectivity of $I_{\R}$
follows basically from the fact that for every positive $n$ the vanishing
locus of the polynomial $f_{n}:=x_{1}^{2}+...+x_{n}^{2}$ in $\R^{n}$ is the
single point $(0_{F},...,0_{F})$. Such a polynomial function can not exist in
algebraically closed fields, but we notice that it exists for any non
algebraically closed field $F,$ which will lead us to the proof of
surjectivity in this case.

\begin{proposition}\label{PolyOnlyVanishZero}
\label{NoAlgClosed} Let $F$ be a non algebraically closed field. Then for
each positive integer $n$ there is a polynomial map $f_{n}:F^{n}\rightarrow
F $ such that $f_{n}^{-1}(0_{F})=\left\{ (0_{F},...,0_{F})\right\} $.
\end{proposition}

\begin{proof} Since $F$ is not algebraically closed there is a monic
polynomial $f(x)\in F[x]$, of positive degree $m$, which has no zeros in $F$
. Let $f_{2}(x,y):=y^{m}f(\frac{x}{y})\in F[x,y]$ be the homogenization of $
f $. Since $f(x)$ has no roots in $F$ then $f_{2}^{-1}(0_{F})=\left\{
(0_{F},0_{F})\right\} $. Now proceed by induction. Let $f_{1}:F\rightarrow F$
be the identity function, and suppose that for $n\geq 1$ we have defined $
f_{n}$. Then $\displaystyle 
f_{n+1}(x_{1},...,x_{n+1}):=f_{2}(f_{n}(x_{1},...,x_{n}),x_{n+1})$ is also
polynomial and
\begin{equation*}
f_{n+1}^{-1}(0_{F})=(f_{n} \times f_{1})^{-1}(\left\{ (0_{F},0_{F})\right\}
)=\left\{ (0_{F},...,0_{F})\right\} \times \{0_{F}\}=\left\{
(0_{F},...,0_{F})\right\} .
\end{equation*} \end{proof}

For the complex numbers surjectivity of $I_{\C}$ may be shown as for $I_{\R}$
by taking the continuous function $f_{n}:=x_{1}\overline{x_{1}}+...+x_{n} 
\overline{x_{n}}$. It is possible that for any field there is continuous
function vanishing only at $(0_{F},...,0_{F})$ of the form $\displaystyle 
f=x_{1}\phi _{1}+...+x_{n}\phi _{n},$ where $x_{i}:F^{n}\rightarrow F$
denotes the $i$-th coordinate function and $\phi _{i}:F^{n}\rightarrow F$ is
continuous$.$ However, we have not been able to prove this. Fortunately, a
weaker version of this idea gives an alternative argument that works for all
fields.

\begin{theorem}
\label{surjective}Let $F$ be a topological field. Then for any compact
Hausdorff space $X$ the Gelfand map $I_{F}:X\rightarrow \mathrm{Max}(C(X,F))$
is surjective.
\end{theorem}

\begin{proof} Let $M\in \mathrm{Max}(C(X,F))$ and for $\psi \in M$ let $
D(\psi ):=\psi ^{-1}(F\setminus \{0_{F}\})$. Notice that $
I_{F}^{-1}(M)=\{x:I_{F}(x)\supseteq M\}=\bigcap_{\psi \in M}\psi
^{-1}(0_{F}).$ Hence, 
\begin{equation*}
I_{F}^{-1}(M)=X\setminus \bigcup_{\psi \in M}D(\psi ).
\end{equation*}
Suppose that $M$ is not in the image of $I_{F}$. It follows from the
compactness of $X$ that there are finitely many $\psi _{1},...,\psi _{n}\in
M $ such that $X=D(\psi _{1})\cup ...\cup D(\psi _{n})$. Hence, any element
is not zero by some $\psi _{i}.$ We will show that there is a continuous function $
f_{n}: F^{n} \to F$ such that $f_{n}(\psi _{1},...,\psi _{n})\in M$
and $f_{n}(\psi _{1}(x),...,\psi _{n}(x))$ is never 0 in $X$. This
contradicts the fact that $M$ is a proper ideal. To achieve this we consider
several cases.

\ \ 

\textbf{Case I}. $F$ is non-discrete locally compact then $F$ must be $
\mathbb{C}$, $\mathbb{R}$, a finite extension of $\mathbb{Q}_{p}$ or a
finite extension of $\mathbb{F}_{p}((t)).$ For $\mathbb{C}$ take $f_{n}:=x_{1}\overline{x_{1}}+...+x_{n}
\overline{x_{n}}$. In the remaining cases $F$ is not
algebraically closed. 
Therefore, the polynomial $f_{n}$ provided by Proposition \ref{PolyOnlyVanishZero} will do: $
f_{n}(\psi _{1}(x),...,\psi _{n}(x))$ never vanishes because $(\psi
_{1}(x),...,\psi _{n}(x))$ is never $(0_{F},...,0_{F})$ and $f_{n}(\psi
_{1},..,\psi _{n})\in M$ since $f_{n}$ does not have constant term.

\ \ 

\textbf{Case II.} $F$ is discrete$.$ Consider the continuous map $x\mapsto
(\psi _{1}(x),...,\psi _{n}(x))\ $from $X$ into $F^{n}.$ Its image $
J=\{(\psi _{1}(x),...,\psi _{n}(x)):x\in X\}$ is compact and thus necessarily
finite since $F^{n}$ is discrete, moreover, it does not include $
(0_{F},...,0_{F})$ by hypothesis. Find a polynomial $f_{n}(x_{1},...,x_{n})$
identically $1_{F}$ in $J$ and $0_{F}$ in $(0_{F},...,0_{F})$ (Lagrange
interpolation). Then $f_{n}(\psi _{1}(x),...,\psi _{n}(x))=1_{F}$ for any $
x\in X$ and belongs to $M$ because $f_{n}$ does not have constant term.

\ \ 

\textbf{Case III}. $F$ is not locally compact. Consider first $D(\psi
_{1})\cup D(\psi _{2})$ restricted to the compact space $Y=X\smallsetminus
D(\psi _{3})\cup ...\cup D(\psi _{n}),$ and consider the map $x\mapsto
\lbrack \psi _{1}(x),\psi _{2}(x)]\ $from $Y$ into $P_{1}(F),$ the
projective $F$-line with the quotient topology, and let $I$ be the image of
this map. Recall that $P_{1}(F)=F\cup \{[1,0]\}$ where $F$ is topologically
embedded in $P_{1}(F)$ via the identification $F\sim \{[a,1]:a\in F\}$. As $
I $ is compact, the set $P_{1}(F)\smallsetminus I$ must be infinite, as otherwise $
P_{1}(F)$ would be compact and thus $F$ would be locally compact. Then we
may find $[a,1]\in F$ such that $[\psi _{1}(x),\psi _{2}(x)]\neq \lbrack
a,1]\ $for each $x\in Y.$ If $\psi _{2}(x)=0$ then $\psi _{1}(x)\neq 0$ and
thus $\psi _{1}(x)-a\psi _{2}(x)\neq 0$. If $\psi _{2}(x)\not=0$ then $\psi
_{1}(x)/\psi _{2}(x)\neq a$ and also $\psi _{1}(x)-a\psi _{2}(x)\neq 0.$
Therefore $\psi _{1}-a\psi _{2}$ does not vanish in $Y$ and we have 
\begin{equation*}
X=D(\psi _{1}-a\psi _{2})\cup D(\psi _{3})\cup ...\cup D(\psi _{n})
\end{equation*}%
Repeating the procedure with the pair $[\psi _{1}-a\psi _{2},\psi _{3}],$
and continuing inductively, we obtain finally $X=D(a_{1}\psi _{1}+a_{2}\psi
_{2}+...+a_{n}\psi _{n})$ where evidently $a_{1}\psi _{1}+a_{2}\psi
_{2}+...+a_{n}\psi _{n}\in M.$ \end{proof}

\section{Conclusion}

Thanks to theorems \ref{main} and \ref{surjective}, we have our main result:

\begin{theorem}
\label{TheRealMain}A topological field $F$ is path connected if and only if
for every compact Hausdorff topological space $X$ the Gelfand map $%
I_{F}:X\rightarrow \mathrm{Max}(C(X,F))$ is a homeomorphism.
\end{theorem}

An important question to ask at this point is: are there other examples
besides $\R$ and $\C$ that make Gelfand correspondence work? As it turns
out there are many other path connected topological fields of 0
characteristic not isomorphic to either $\R$ or $\C$. There are even
examples in positive characteristic; see for instance \cite{WaBe:1966},
where it is shown that any discrete field may be embedded in a path
connected field. However, by Pontryagin's classification theorem\ \cite%
{Pontrjagin:1932} the only locally compact ones are $\R$ and $\C$.

\ \ \ 

An inspection of the proof of Theorem \ref{surjective} shows that for any
topological field $F,$ given functions $\psi _{1},...,\psi _{n}\in C(X,F)$
such that $X=D(\psi _{1})\cup ...\cup D(\psi _{n})$ there exist continuous
functions $\phi _{1},...,\phi _{m}:F^{n}\rightarrow F$ such that $%
\displaystyle\Sigma _{i=1}^{n}\psi _{i}\phi _{1}(\psi _{1},..,\psi _{n})$
does not vanishes in $X.$ But the $\phi _{i}$ depend on the $\psi _{i}$. One
may wonder if the $\phi _{i}$ may be chosen independently; that is, whether
there exists $f=x_{1}\phi _{1}+...+x_{n}\phi _{n}$ such that $%
f^{-1}(0_{F})=(0_{F},...,0_{F}).$ We call this \textit{polynomially
generated } functions. We are able to show:

\begin{proposition}
\label{TeoTheMap} Let $F$ be a path connected metrizable topological field.
Then for every positive integer $n$ there is a continuous function $%
f_{n}:F^{n}\rightarrow F$ such that $f_{n}^{-1}(0_{F})=\left\{
(0_{F},...,0_{F})\right\} .$
\end{proposition}

\begin{proof} Let $d$ be a metric on $F^{n}$ and let $\bar{d}:=\min\{1, d\}$  be the standard bounded metric on  $F^{n}$. Define $\phi :F^{n}\rightarrow
\lbrack 0,1]$ as the bounded distance to $(0_{F},...,0_{F}).$ As $F$ is
arcwise connected (Lemma \ref{conexoCharCero}), there is an embbeding $[0,1]
\overset{\gamma }{\rightarrow }F$ sending $0$ to $0_{F}.$ Take $f_{n}=\gamma
\circ \phi .$ \end{proof}

We finish with the following question, that appeared naturally during the
progress of the paper, for which we don't have an answer.

\begin{question}
Let $F$ be a path connected topological field. Is there is a polynomially
generated function $f_{2} :F^{2}\rightarrow F$ such that $\displaystyle %
f_{2}^{-1}(\{0_{F}\})=(0_{F},0_{F})$?
\end{question}


\noindent

{\footnotesize Xavier Caicedo, Department of Mathematics, Universidad de los
Andes, Bogot\'a, Colombia (\texttt{xcaicedo@uniandes.edu.co})}

\noindent

{\footnotesize Guillermo Mantilla-Soler, Department of Mathematics,
Universidad Konrad Lorenz, Bogot\'a, Colombia (\texttt{gmantelia@gmail.com})}

\end{document}